\documentclass{conm-p-l}

\usepackage{}
\usepackage{trajan}
\usepackage{multirow}
\usepackage{latexsym,array,delarray,amsthm,amssymb,epsfig}
\usepackage{amsmath}

\theoremstyle{plain}
\newtheorem{thm}{Theorem}[section]
\newtheorem{lemma}[thm]{Lemma}
\newtheorem{prop}[thm]{Proposition}
\newtheorem{cor}[thm]{Corollary}

\newtheorem*{thm*}{Theorem}
\newtheorem*{lemma*}{Lemma}
\newtheorem*{prop*}{Proposition}
\newtheorem*{cor*}{Corollary}
\newtheorem*{conj*}{Conjecture}

\theoremstyle{definition}
\newtheorem{defn}[thm]{Definition}
\newtheorem{ex}[thm]{Example}

\theoremstyle{remark}
\newtheorem*{rmk}{Remark}

\numberwithin{equation}{section}

\newcommand{\ff}{\mathbb{F}}

\newcommand{\A}{\bf A}

\begin{document}

\title[LEIBNIZ ALGEBRAS]{ON SOME STRUCTURES  OF LEIBNIZ ALGEBRAS}


\author{Ismail Demir, Kailash C. Misra and Ernie Stitzinger}
\address{Department of Mathematics, North Carolina State University, Raleigh, NC 27695-8205}
\email{idemir@ncsu.edu, misra@ncsu.edu, stitz@ncsu.edu}
\thanks{KCM is partially supported by NSA grant \#  H98230-12-1-0248}


\subjclass[2010]{17A32 , 17A60 }

\date{04/02/2013}

\begin{abstract}
Leibniz algebras are certain generalization of Lie algebras. In this paper we survey the important results in Leibniz algebras which are analogs of corresponding results in Lie algebras. In particular we highlight the differences between Leibniz algebras and Lie algebras.
\end{abstract}

\maketitle
\bigskip
\section{Introduction}
\par
Leibniz algebras were introduced by Loday in \cite{lodayfr} as a noncommutative generalization of Lie algebras. Earlier, such algebraic structures had been considered by Bloh who called them $D$-algebras \cite{bloch}. In studying the properties of the homology of Lie algebras, Loday observed that the antisymmetry of the product was not needed to prove the derived property defined on chains. This motivated him to introduce the notion of right (equivalently, left) Leibniz algebras, which is a nonassociative algebra with the right (equivalently, left) multiplication operator being a derivation. Thus a Leibniz algebra satisfies all defining properties of a Lie algebra except the antisymmetry of its product.
\\
\par
Since the introduction of Leibniz algebras around 1993 several researchers have tried to find analogs of important theorems in Lie algebras. 
For example, it is now known that there are analogs of Lie's Theorem, Engel's Theorem, Cartan's criterion, and Levi's Theorem for Leibniz algebras. 
However, some of these results are proved for left Leibniz algebras and some are proved for right Leibniz algebras in the literature. In this paper following Barnes \cite{firstpaper} we focus on left Leibniz algebras and state all important known results for these algebras. If some result is known only for right Leibniz algebras, then (if necessary) we include the modified proof to show that they hold for left Leibniz algebras. We include examples to show that certain results in Lie algebras do not have analogs in Leibniz algebras. 
Throughout this paper unless indicated otherwise all algebras and modules are assumed to be finite dimensional over  an algebraically closed field $\ff$ with characteristic zero.
\\
\par
The paper is organized as follows. In Section 2 we introduce the basic concepts for Leibniz algebras and give examples to illustrate their importance. In particular, we define the modules and representations for Leibniz algebras. In Section 3 we focus on solvable Leibniz algebras and give analogue of Lie's Theorem and Cartan's criterion. We also state the analogue of Levi's Theorem. We consider the nilpotent Leibniz algebras in Section 4. We prove a version of Engel's Theorem using Lie subset and standard results in Lie algebra theory and obtain the analogue of the classical Engel's Theorem as a consequence. We also define Engel subalgebra for a Leibniz algebra following Barnes \cite{firstpaper} and show that a minimal Engel subalgebra is a Cartan subalgebra. In Section 5 we define semisimple Leibniz algebra and an analogue of the Killing form for Leibniz algebras. In particular we show that if the Leibniz algebra is semisimple then the Killing form is nondegenerate, but the converse is not true. These results seem to be new. In the last section (Section 6) we focus on classifications of non-Lie Leibniz algebras of dimensions $2$ and $3$. In particular, to classify three dimensional nilpotent Leibniz algebras we use a new approach involving the canonical forms for the congruence classes of matrices for bilinear forms which can easily be used to classify higher dimensional nilpotent Leibniz algebras. We intend to continue this investigation in near future.

\section{Preliminaries}

In this section we recall the basic definitions for Leibniz algebras and their representations. A (left) Leibniz algebra (or Leibniz algebra) $\A$ is a $\ff$-vector space equipped with a bilinear map (multiplication)
$$
[\, , \,] : {\A} \times {\A} \longrightarrow {\A}
$$
satisfying the Leibniz identity
\begin{equation}\label{leibnizid}
[a, [b, c]]=[[a, b], c]+[b, [a, c]] 
\end{equation} 
for all $a, b, c\in \A$.

For a Leibniz algebra $\A$ and $a \in \A$ , we define the left multiplication operator $L_a : \A \longrightarrow \A$ and the right multiplication operator $R_a : \A \longrightarrow \A$ by $L_a(b)=[a,b]$ and $R_a(b)=[b,a]$ respectively for all $b\in \A$. Note that by (\ref{leibnizid}), the operator $L_a$ is a derivation, but $R_a$ is not a derivation. Furthermore, the vector space $L({\A}) = \{L_a \mid a \in {\A}\}$ is a Lie algebra under the commutator bracket. A right Leibniz algebra is a vector space equipped with a bilinear multiplication such that the right multiplication operator is a derivation. Throughout this paper Leibniz algebra always refers to (left) Leibniz algebra. As the following example shows a (left) Leibniz algebra is not necessarily a (right) Leibniz algebra.

\begin{ex} Let $\A$ be a $2$-dimensional algebra with the following multiplications.
\begin{center} $[x, x]=0, [x, y]=0, [y, x]=x, [y, y]=x$ 
\end{center}
$\A$ is a (left) Leibniz algebra, but it is not a (right) Leibniz algebra, since $[[y, y], y]\neq [y, [y, y]]+[[y, y], y]$. 
\end{ex}

Any Lie algebra is clearly a Leibniz algebra. A Leibniz algebra $\A$ satisfying the condition that $[a, a]=a^2=0$ for all $a\in \A$, is a Lie algebra since in this case the Leibniz identity becomes the Jacobi identity. Given any Leibniz algebra $\A$ we denote $Leib({\A}) = {\rm span}\{[a,a] \mid a \in {\A}\}$. Any associative algebra $\A$ equipped with an idempotent linear operator can be given the structure of a Leibniz algebra as follows (see \cite{loday}).

\begin{ex} Let $\A$ be an associative $\ff$-algebra equipped with a linear operator 
$T:{\A}\rightarrow {\A}$ such that  $T^2=T$. Define the multiplication $[\, , \,] : {\A} \times {\A} \longrightarrow {\A}$ by
\begin{center} $[a, b]:=(Ta)b-b(Ta)$ for all $a, b\in \A$.
\end{center}
It is easy to see that the Leibniz identity (\ref{leibnizid}) holds. So $\A$ becomes a Leibniz algebra. 
\end{ex}
\noindent
Observe that in the above example
$\A$ is a Lie algebra if and only if $T=id$.
\\

Let $\A$ be a Leibniz algebra. Then the Leibniz identity (\ref{leibnizid}) implies the following identities.

\begin{equation}\label{LRidentity}
\begin{cases}
R_{[b, c]}=R_cR_b+L_bR_c \\
L_bR_c=R_cL_b+R_{[b, c]}  \\       
L_cL_b=L_{[c, b]}+L_bL_c  
\end{cases}
\end{equation}

Hence using the first two equations in (\ref{LRidentity}), we obtain $R_cR_b=-R_cL_b$ for all $b, c\in \A$ which implies that for $n \geq 1$, $R_a^n=(-1)^{n-1}R_aL_a^{n-1}$ for $a\in \A$. Therefore, $R_a$ is nilpotent if $L_a$ is nilpotent. For any element $a \in \A$ and $n \in \mathbb{Z}_{>0}$ we define $a^n \in \A$ inductively by defining $a^1=a$ and $a^{k+1} = [a,a^k]$. Similarly, we define ${\A}^n$ by ${\A}^1= \A$ and ${\A}^{k+1} = [\A , \A^k]$. The Leibniz algebra $\A$ is said to be abelian if ${\A}^2 = 0$. Furthermore, it follows from (\ref{leibnizid}) that $L_{a^n} = 0$ for $n \in \mathbb{Z}_{>1}$.

\begin{ex} \label{cyclic} Let $\A$ be a $n$-dimensional Leibniz algebra over $\ff$ generated by a single element $a$. Then ${\A} = {\rm span}\{a, a^2, \cdots , a^n\}$ and we have  $[a, a^n] = \alpha_1 a + \cdots + \alpha_n a^n$ for some $\alpha_1, \cdots \alpha_n \in \ff$. By Leibniz identity we have $ 0 = [a , [a^n , a]] = [[a,a^n], a] + [a^n , [a , a]] = [ \alpha_1 a + \cdots + \alpha_n a^n , a] = \alpha_1 [a , a]$ which implies that $\alpha_1 = 0$. Hence ${\A}^2 = {\rm span}\{a^2, \cdots , a^n\} = Leib({\A})$. The Leibniz algebra $\A$ is called a $n$ - dimensional cyclic Leibniz algebra.
\end{ex}

Let  ${I}$ be a subspace of a Leibniz algebra $\A$. Then $I$ is a subalgebra if $[I,I] \subseteq I$, a left (resp. right) ideal if $[{\A} , I] \subseteq I$ (resp. $[I , \A] \in I$). $I$ is an ideal of $\A$ if it is both a left ideal and a right ideal.  In such case we denote  $I\lhd \A$. In particular, $Leib({\A})$ is an abelian ideal of $A$. By definition $Leib({\A})$ is a right ideal. The fact that $Leib({\A})$ is a left ideal follows from the identity 
$[a , [b , b]] = [a + [b , b] , a + [b ,b]] - [a , a]$. For any ideal $I \lhd \A$ we define the quotient Leibniz algebra in the usual way. In fact, $Leib({\A})$ is the minimal ideal such that ${\A} / Leib({\A})$ is a Lie algebra.  As in case of Lie algebras, the sum and intersection of two ideals of a Leibniz algebra is an ideal. However, the product of two ideals need not be an ideal as shown below (see  \cite{schunck}).

\begin{ex}  Let ${\A} = {\rm span}\{x, a, b, c, d\}$ with multiplications $[a, b]=c, [b, a]=d,  [x, a]=a=-[a, x], [x, c]=c, [x, d]=d, [c, x]=d, [d, x]=-d, $ and the rest are zero. Let $I={\rm span}\{a, c, d\}$ and $J={\rm span}\{b, c, d\}$. Then $I, J$ are ideals of $\A$, but $[I, J]= {\rm span}\{c\}$ which is not an ideal.
\end{ex}

Given two Leibniz algebras $\A$ and $\A'$, a linear map $\varphi : {\A} \longrightarrow {\A'}$ is called a homomorphism if it preserves the multiplications. The kernel of a homomorphism is an ideal and the standard homomorphisms for Lie algebras also hold for Leibniz algebras. In particular, for any ideal $I\lhd \A$, the ideals of the quotient Leibniz algebra ${\A}/ I$ are in one-to-one correspondence with the ideals of $\A$ containing $I$.  

The left center of $\A$ is denoted by $Z^{l}({\A})$= \{$x\in {\A} \mid [x, a]=0$ for all $a\in {\A}$\} and the right center of $\A$ is denoted by $Z^{r}({\A})$= \{$x\in {\A} \mid [a, x]=0$ for all $a\in {\A}$\}. The center of $\A$ is $Z({\A}) = Z^{l}({\A}) \cap Z^{r}({\A})$. Let $H$ be a subalgebra of the Leibniz algebra $\A$. The left normalizer of $H$ in $\A$ is defined by $N^l_{\A}(H)=\{x\in {\A} \mid [x, a]\in H$ {\rm for \, all} $a\in H$\}. The right normalizer is defined by $N^r_{\A}(H)=\{x\in {\A} \mid [a, x]\in H$ {\rm for \, all} $a\in H$\}. The normalizer is defined by $N_{\A}(H) = N^{l}_{\A}(H)\cap N^{r}_{\A}(H)$. It is easy to see that the normalizer $N_{\A}(H)$ and the left normalizer $N^l_{\A}(H)$ are both subalgebras, but the right normalizer $N^r_{\A}(H)$ need not be a subalgebra as is shown by the following example (see \cite{firstpaper}).

\begin{ex} Let ${\A}= {\rm span}\{u, n, k, n^2\}$ with the multiplication given by $[u, n]=u, [n, u]=-u+k, [u,n^2]=k, u^2=0, [u, k]=0, [n, k]=-k, n^3=0$ and $[k, a]=[n^2, a]=0$ for all $a\in {\A}$. Let $H=<u>$. Then $\A$ is a Leibniz algebra, $H$ is a subalgebra and $N^r_{\A}(H)$ is not a subalgebra of $\A$, as $n\in N^r_{\A}(H)$ but $n^2\notin N^r_{\A}(H)$.
\end{ex}

\begin{defn} A module of a Leibniz algebra $\A$ is a vector space $M$ with two bilinear maps $[ \, , \, ] : {\A} \times M \longrightarrow M$ and 
$[ \, , \, ] : M \times {\A} \longrightarrow M$  such that 
\begin{center}
 $[a, [b, m]]=[[a, b], m]+[b, [a, m]]$ 
\par $[a, [m, b]]=[[a, m], b]+[m, [a, b]]$
\par$[m, [a, b]]=[[m, a], b]+[a, [m, b]]$  
\end{center}    
 for all $a, b\in {\A}$  and $m \in M$.

\end{defn}

Let $End(M)$ denote the associative algebra of all endomorphisms of the vector space $M$. If $M$ is a $\A$-module, then each of the maps  $T_a:m\longrightarrow [a, m]$ and   $S_a:m\longrightarrow [m, a]$ is an endomorphism of $M$. Also the maps $T_a:a\longrightarrow T_a, S_a:a\longrightarrow S_a$ from $\A$ into $End(M)$ are linear.

\begin{defn} A representation of a Leibniz algebra $\A$ on a vector space $M$ is a pair $(T, S)$ of linear maps $T: {\A}\longrightarrow End(M)$ with $T(a)=T_a$, $S: {\A} \longrightarrow End(M)$ with $S(a)=S_a$ such that
\begin{center}
 $S_b\circ S_a=S_{[a, b]}-T_a\circ S_b$
\\
$S_b\circ T_a=T_a\circ S_b-S_{[a, b]}$
\\
$T_{[a, b]}=T_a\circ T_b-T_b\circ T_a$
\end{center}
for all $a, b\in A$.

\end{defn}

We say that the pair $(T, S)$ is the associated representation of the $\A$ - module $M$.
We denote $Ker(T, S)=\{a\in A \mid T_a=0=S_a\}=Ker(T)\cap Ker(S)$. Then we have the following result which we need later.

\begin{lemma}\cite{firstpaper} \label{repn} Let $M$ be an irreducible module for the Leibniz algebra ${\A}$ and let $(T, S)$ be the associated representation. Then ${\A}/Ker(T, S)$ is a Lie algebra, and either $S_a=0$ or $S_a=-T_a$ for all $a\in \A$.
\end{lemma} 

\section{Solvability}
In this section we focus on solvable Leibniz algebras and give the analogs of Lie's Theorem, Cartan's Criterion and Levi's Theorem for Leibniz algebras.
Let $\A$ be a Leibniz algebra. Then the series of ideals
\begin{center}
$A\supseteq {\A}^{(1)}\supseteq {\A}^{(2)}\supseteq\ldots   $ where  ${\A}^{(1)}=[{\A}, {\A}], {\A}^{(i+1)}=[{\A}^{(i)}, {\A}^{(i)}]$ 
\end{center}
is called the derived series of $\A$. 

\begin{defn} A Leibniz algebra $\A$ is solvable if ${\A}^{(m)}=0$ for some integer $m\geq 0$.
\end{defn}

It is easy to see that the Leibniz algebra $\A$ is solvable if and only if the associated Lie algebra $L({\A})$ of derivations of $\A$ is solvable.
As in the case of Lie algebras, the sum and intersection of two solvable ideals of a Leibniz algebra is solvable. Hence any Leibniz algebra $\A$ contains  a unique maximal solvable ideal $rad({\A})$ called the radical of $\A$ which contains all solvable ideals. 
We give the following analogue of Lie's theorem which is probably known. 
\begin{thm} \label{Lie}Let $\A$ be a solvable  Leibniz algebra over $\ff$ and $M$ be an irreducible $\A$-module. Then ${\rm dim}(M) =1$. 
\end{thm}
\begin{proof} Since $\A$ is a solvable Leibniz algebra we have the associated Lie algebra $L({\A})$ solvable. By Lie's theorem for Lie algebras, there exists a nonzero $m\in M$ such that $m$ is $L({\A})$-invariant. Since $M$ is an irreducible $\A$- module, we have by Lemma \ref{repn}, $[m, a]=0$ or $[a, m]=-[m, a]$ for all $a\in \A$. In either case we have $[a, m], [m, a]\in {\rm span}\{m\}$ and ${\rm span}\{m\}$ is an $\A$ - submodule of $M$. Since $M$ is irreducible, it follows that $M = {\rm span}\{m\}$. Therefore, we have ${\rm dim}(M) =1$.
\end{proof}

\begin{cor} \label{flag} Let $\A$ be a $n$-dimensional solvable Leibniz algebra over $\ff$. Then there is a chain of ideals
\begin{center} 
$0=I_0\subseteq I_1\subseteq I_2\subseteq \ldots \subseteq I_n=\A$
\end{center} 
of $\A$ such that $dim(I_i)=i$ for $i=1, 2, \dots ,n$.
\end{cor}
\begin{proof} We will prove it by the induction on $dim({\A})$. If $\A$ is $1$-dimensional then the statement holds. Let $J$ be a minimal ideal. Then $J$ is an irreducible $\A$- module. Hence by Theorem \ref{Lie}, 
${\rm dim}(J) = 1$. Now by the induction hypothesis the $(n-1)$- dimensional Leibniz algebra ${\A}/J$ has a chain of ideals 
\begin{center} $0=I_1/J\subseteq I_2/J \subseteq I_3/J \subseteq \ldots \subseteq I_{n}/J={\A}/J$
\end{center}
such that $dim(I_i/J)=i-1$ for $i=1, 2,\ldots ,n$. Then we have the desired chain of ideals
\begin{center} 
$0=I_0  \subseteq I_1=J \subseteq I_2\subseteq \ldots \subseteq I_n=\A$
\end{center}  
of $\A$ with $dim(I_i)=i$ for $i=1, 2, \dots ,n$.
\end{proof}
The following result now follows from Corollary \ref{flag}.

\begin{cor} \label{uptri}Let $\A$ be a solvable Leibniz algebra. Then the left multiplication operators $\{L_x \mid x \in {\A}\}$ can be simultaneously upper triangularized.
\end{cor}
Now we have the following analogue of Cartan's criterion for Leibniz algebras which is given in \cite{weight}. We include an easy proof using the  corresponding Lie algebra result.

\begin{thm}\label{Cartan} Let $A$ be a Leibniz algebra over $\ff$. Then $A$ is solvable if and only if $tr(L_aL_b)=0$ for all $a\in A^2$ and all $b\in A$.
\end{thm}

\begin{proof} Let $\A$ be solvable. Then  $L({\A})$ is a solvable Lie algebra of linear operators on $\A$. Hence by Lie's Theorem the left multiplication operators $L_x$ can be simultaneously represented by matrices in upper triangular form. Then  the matrix of 
$L_{[x, y]}=[L_x, L_y]$ is in strictly upper triangular form as that of  $L_{[x, y]}L_z$. Hence $tr(L_{[x, y]}L_z)=0$ for all $x, y, z\in \A$.
\\
 \par
Conversely, if the condition holds, then $tr(ST)=0$ for all $S\in [L({\A}), L({\A})]$ and $T\in L({\A})$. Then $L({\A})$ is solvable by Cartan's criterion for Lie algebras. Hence $\A$ is solvable.
\end{proof}

We also have the following analogue of Levi's Theorem for Leibniz algebras. However, as shown in \cite{barneslevi} Malcev's congugacy theorem does not hold for Leibniz algebras.

\begin{thm} \cite{barneslevi} \label{levi} Let $\A$ be a Leibniz algebra over $\ff$  and $R = rad({\A})$ be its solvable radical. Then there exists a subalgebra $S$ (which is a semisimple Lie algebra) of $\A$ such that ${\A}=S\dotplus R$.
\end{thm}

\section{Nilpotency}

In this section we show the Leibniz algebra version of Engel's theorem and some of it's consequences.  We also introduce Engel subalgebras and use them to give a characterization of nilpotency and to show existence of Cartan subalgebras.

\vspace{5mm}
\begin{defn}
A Leibniz algebra $\A$  is nilpotent of class c if each product of c+1 elements is 0 and some product of c elements is not 0. An element is left normed if it is of the form  $[a_1,[a_2,[...[a_{n-1},a_n]...]]]$. It is convenient to write such an element as  $[a_1,a_2,...,a_n]$.
\end{defn}
\vspace{5mm}

\begin{prop}\label{normed} Any element in the Leibniz algebra $\A$ that is the product of $n$ elements can be expressed as a linear combination of the $n$ elements with each term being left normed.
\end{prop}

\begin{proof} We use induction on $n$. If $n=3$, then by Leibniz identity $[[a,b],c]=[a,[b,c]]-[b,[a,c]]$. Assume the result for products with less than $n$ elements. Any product with $n$ elements is of the form $[r,s]$ where $r$ contains $i$ elements and $s$ contains $j=n-i$ elements. By induction, $r$ is a linear combination of left normed elements, each of which is of the form $[a,t]$ where $a$ is a single element and $t$ is left normed with $j-1$ elements. Then $[r,s] = [[a,t],s] = [a,[t,s]]-[t,[a,s]]$. In the first term, $[t,s]$ can be written in the desired form by induction and then so can $[a,[t,s]]$. In the second term $[a,s]$ can be written in the desired form by induction. Furthermore, in the second term  the first argument has one less term than $r$ and the second part $[a,s]$ has one more term than s. Repeat the process on $[t,[a,s]]$. If we continue this process, after finitely many steps the first arguments will reduce to just one term which completes the proof. 
\end{proof}

For a Leibniz algebra  $\A$ the series of ideals
\begin{center}
${\A}\supseteq {\A}^{1} \supseteq {\A}^{2} \supseteq \ldots$    where $ {\A}^{i+1}=[{\A}, {\A}^{i}]$
\end{center} 
( as defined before) is called the lower central series of $\A$.
As an immediate consequence of Proposition \ref{normed}, we have


\begin{cor}
A Leibniz algebra is nilpotent of class $c$ if A$^{c+1}=0$ but A$^{c}\neq 0$.
\end{cor}

Thus if $\A$ is a nilpotent Leibniz algebra of class $c$ then  ${\A}^{c}\subseteq Z^r({\A})$. By Proposition \ref{normed}, we also have ${\A}^{c}\subseteq Z^l({\A})$. Hence  ${\A}^c \subseteq Z^l({\A}) \cap Z^r({\A}) = Z({\A})$ and $Z({\A}) \neq 0$. As in case of Lie algebras, the sum and intersection of two nilpotent ideals of a Leibniz algebra $\A$ are nilpotent. Hence $\A$ contains a unique maximal nilpotent ideal called the nilradical of $A$ and denoted by $nil({\A})$. Let $R = rad({\A})$ and $N = nil({\A})$.
Then as shown in (\cite{vgor}, Proposition 4), $[{\A} , R] \subset N$. So $[R , R]$ is nilpotent. In particular, if $\A$ is solvable, we have $[{\A} , {\A}]$ is nilpotent. In fact, we have the following result.

\begin{prop} \cite{vgor} The Leibniz algebra $\A$ is solvable if and only if $[{\A} , {\A}]$ is nilpotent.
\end{prop}

Several authors have given various forms of Engel's theorem for Leibniz algebras. The Leibniz versions were proven without using the standard Lie algebra result in \cite{omirov1998} and \cite{patsourakos}, and using the standard result in \cite{barnesengel} and \cite{secondpaper}. The second method greatly shortens the proof. The short proof given here shows the result using Lie sets and follows that in \cite{secondpaper}. A Lie subset of a Leibniz algebra $\A$ is a subset that is closed under multiplication. 


\begin{thm} \label{Engel} Let $\A$ be a Leibniz algebra, $L$ be a Lie subset of $\A$ that spans $\A$ and $M$ be an $\A$-module with associated representation $(T,S)$. Suppose that $T_a$ is nilpotent for all $a \in L$.  Then $\A$ acts nilpotently on $M$ and there exists $0 \neq m \in M$  such that $[a,m]=[m,a]=0$ for all $a\in {\A}$.
\end{thm}

\begin{proof} If $M$ is irreducible, then by Lemma \ref{repn} either $[M,A] = 0$ or $[m,a] = -[a,m]$ for all $a \in {\A},  m \in M$. $T({\A})$ is the homomorphic image of $\A$ and is a Lie algebra with Lie set $T(L)$ which spans $T({\A})$ and each element of $T(L)$ acts nilpotently on $M$. Hence there is a common eigenvector $m$, associated with the eigenvalue $0$ for all $a \in {\A}$ by Jacobson's refinement to Engel's theorem (for Lie algebras). Since the right multiplication by all $a \in {\A}$ also has $[m,a] = 0$ by Lemma \ref{repn}, $M$ is the one dimensional space spanned by $m$. For general $M$, let $N$ be an irreducible submodule. Then $[A,N] = [N,A] = 0$. By induction $\A$ acts nilpotently on $M/N$ and hence on $M$.
\end{proof}

The following extensions of classical Engel's theorem results are seen to be corollaries of the above theorem.


\begin{cor}
 Let $\A$ be a Leibniz algebra with Lie set $L$ that spans $\A$. Suppose that $L_a$ is nilpotent for all $a \in L$. Then $\A$ is nilpotent.
 \end{cor}


\begin{cor}
 Let $\A$ be a Leibniz algebra with module $M$. Suppose that $T_a$ acts nilpotently on M for all $a \in {\A}$. Then there is a flag of submodules in which $\A$ annihilates each factor. Hence $\A$ acts nilpotently on $M$.
 \end{cor}


\begin{cor}\label{Engel2}
 Let $\A$ be a Leibniz algebra in which each left multiplication operator is nilpotent. Then $\A$ is nilpotent.
 \end{cor}


Several other corollaries also hold. Jacobson \cite{jacob} has shown that Lie algebras that admit certain operators are nilpotent. The Leibniz algebra extensions are shown using the same methods and hence the proofs are omitted here.


\begin{cor}
 Let $\A$ be a Leibniz algebra that admits an automorphism of prime period with no non-zero fixed points. Then $\A$ is nilpotent.
 \end{cor}


\begin{cor}
Let $\A$ be a Leibniz algebra over $\ff$ that admits a non-singular derivation. Then $\A$ is nilpotent.
\end{cor}


Jacobson asked if there is a converse to the last corollary for Lie algebras. Dixmier and Lister \cite{dix} constructed an example to show the converse does not hold directly. Allowing a weaker form of operator than derivation, Moens (see \cite{moens}) developed a converse. A similar process has been obtained for Leibniz algebras in \cite{charac}. This proceeds as follows. A linear transformation, $D$,  on $\A$ is called a Leibniz derivation of order $s$ if $D([x_1,...,x_s]) = \sum_{j=1}^s ([x_1,...,D(x_j),...,x_s])$.
This concept supplies a converse.


\begin{thm} \cite{charac}  Every nilpotent Leibniz algebra $\A$ of class $c$ admits an invertible Leibniz derivation $D$ of order $\lfloor c/2 \rfloor +1$.
\end{thm}


Now we arrive at the following characterization of nilpotency.


\begin{thm} \cite{charac} A Leibniz algebra over $\ff$ is nilpotent if and only if it has a non-singular Leibniz derivation.
\end{thm}


Engel subalgebras have been introduced by Barnes to show the existence of Cartan subalgebras in Lie algebras (see \cite{1967}) and Leibniz algebras \cite{firstpaper}. They are useful in other contexts also. Let  $\A$ be a Leibniz algebra and $a \in {\A}$. The Fitting null component of the left multiplication operator $L_a$ on $\A$ is called the Engel subalgebra for $a$ and is denoted by $E_{\A}(a)$. It is indeed a subalgebra of $\A$. Unlike in Lie algebras, it is possible that the element  $a$ is not in $E_{\A}(a)$.
%
%
\begin{ex} Let $\A$ be the two dimensional cyclic Leibniz algebra generated by an element $a$ with non-zero product $[a,a^2] = a^2$. Then the Engel subalgebra $E_{\A}(a) = {\rm span} \{a-a^2\}$ and  $a \not\in E_{\A}(a)$.
\end{ex}
%
%
However, we do have the following result.
%
%
\begin{lemma} \cite{firstpaper}. Let $\A$ be a Leibniz algebra. For any $a \in \A$, there exists $b \in E_{\A}(a)$ such that $E_{\A}(a) = E_{\A}(b)$.
\end{lemma}
\begin{lemma}\label{subalg} \cite{firstpaper} Let $\A$ be a Leibniz algebra, $M$ be a subalgebra and $a \in \A$. Suppose $E_{\A}(a) \subseteq M$. Then $M=N_{\A}^r(M)$.
\end{lemma}
%
%
We will next give a characterization of nilpotency for Leibniz algebras. Analogs of these results occur in both group theory and Lie algebras. Let $\A$ be a Leibniz algebra. We say that $\A$ satisfy the normalizer condition if every proper subalgebra of $\A$ is properly contained in it's normalizer. Let $\A$ be nilpotent with proper subalgebra $H$. Let $s$ be the smallest positive integer such that ${\A}^s$ is contained in $H$. Then ${\A}^{s-1}$ is contained in the normalizer of $H$ using Proposition \ref{normed}. Hence nilpotent Leibniz algebras satisfy the normalizer condition. Likewise $\A$ is said to satisfy the right normalizer condition if every proper subalgebra of $\A$ is properly contained in it's right normalizer. If $\A$ satisfies the normalizer condition, then it clearly satisfies the right normalizer condition. Suppose that $\A$ satisfies the right normalizer condition and let $H$ be a proper subalgebra of $\A$. Let $ a \in {\A}$. By Lemma \ref{subalg},  $E_{\A}(a)= N_{\A}^r(E_{\A}(a))$. Hence by the right normalizer property we have $E_{\A}(a)=A$ for all $a\in {\A}$ which implies that the left multiplication operators $\{L_a \mid a \in {\A}\}$ are nilpotent by definition of Engel subalgebras. Therefore $\A$ is nilpotent by Corollary \ref{Engel2}.  Hence the normalizer condition characterizes nilpotency in Leibniz algebras, as does the right normalizer condition. If $\A$ is nilpotent and $H$ is a maximal subalgebra of $\A$, the normalizer condition yields that $H$ is an ideal in $\A$. Clearly if all maximal subalgebras are ideals, then they are right ideals. Finally suppose that all maximal subalgebras of $\A$ are right ideals of $\A$. If $a \in {\A}$ and $E_{\A}(a) \neq {\A}$, then there is a maximal subalgebra, $M$, which contains $E_{\A}(a)$. Again by Lemma \ref{subalg}, $M$ is it's own right normalizer and hence not a right ideal which is, a contradiction. Therefore, we have  $E_{\A}(a)=A$ for all $a \in {\A}$ and $\A$ is nilpotent by Corollary \ref{Engel2}. Thus we have the following theorem parts of which are given in \cite{firstpaper}.
%
%
\begin{thm} Let $\A$ be a Leibniz algebra. Then the following are equivalent:

\begin{enumerate}

\item  $\A$ is nilpotent.

\item  $\A$ satisfies the normalizer condition.

\item $\A$ satisfies the right normalizer condition.

\item Every maximal subalgebra of $\A$ is an ideal of $\A$.

\item Every maximal subalgebra of $\A$ is a right ideal of $\A$.
\end{enumerate}
\end{thm}

Engel subalgebras are useful in showing the existence of Cartan subalgebras. A subalgebra $H$ of a Leibniz algebra $\A$ is called a Cartan subalgebra if $H$ is nilpotent and is $N_{\A}(H) = H$. Barnes \cite{firstpaper} gave the following realization of a Cartan subalgebra for $\A$ which proves its existence.

\begin{thm}\label{cartan} \cite{firstpaper}
A subalgebra of a Leibniz algebra $\A$ is a Cartan subalgebra if it is minimal in the set of all Engel subalgebras of $\A$.
\end{thm}

For Leibniz algebras, the left and right normalizers of a Cartan subalgebra behave differently. The right normalizer of a Cartan subalgebra is equal to the Cartan subalgebra using Lemma \ref{subalg} and the fact that by Theorem \ref{cartan} a Cartan subalgebra is an Engel subalgebra. However, the left normalizer may not equal to the Cartan subalgebra as is seen in the following example.

\begin{ex}
Let $\A$ be the two dimensional cyclic Leibniz algebra generated by an element $a$ with non-zero product $[a,a^2] = a^2$. Then $C=E_{\A}(a)={\rm span}\{a-a^2\}$ is nilpotent and self-normalizing, hence $C$ is a Cartan subalgebra of $\A$. However, the left normalizer $N^l_{\A}(C) = {\A}$ and the right normalizer $N^r_{\A}(C) = C \neq {\A}$.
\end{ex}

\section{Semisimplicity}
In this section we define simple and semisimple Leibniz algebras and discuss some of their important properties. In particular we define the notion of Killing form and show that it is nondegenerate if the Leibniz algebra $\A$ is semisimple, but the converse is not true. It is important to note that in the literature some authors have used different definitions for simple and semisimple Leibniz algebras.
\begin{defn} 
A Leibniz algebra $\A$ is simple if ${\A}^2\neq Leib(A)$ and $\{0\}, Leib({\A}),$  ${\A}$  are the only ideals of $\A$.
\end{defn}

\begin{defn} A Leibniz algebra $\A$ is said to be semisimple if $rad({\A}) = Leib({\A})$. 
\end{defn}
Thus the Leibniz algebra $\A$ is semisimple if and only if the Lie algebra ${\A}/Leib({\A})$ is semisimple. However, if ${\A}/Leib({\A})$ is a simple Lie algebra then $\A$ is not necessarily a simple Leibniz algebra. Also ${\A}/Leib({\A})$ is a semisimple Lie algebra does not imply that $\A$ can be written as direct sum of simple Leibniz ideals as shown in the following example.

\begin{ex}
Consider the simple Lie algebra $sl(2, \mathbb{C})$  and its irreducible module $V(m) = {\rm span}\{v_0, v_1, \cdots v_m\}, m\geq 1$. The actions of $sl(2, \mathbb{C})$ on $V(m)$ is well known (for example see \cite{H}). Consider the algebra ${\A} = sl(2, \mathbb{C}) \dotplus V(m)$ with the left multiplication of any vector in $sl(2, \mathbb{C})$ with a vector in $V(m)$ given by the module action and the right multiplication being trivial. Then as shown in \cite{sl2} $\A$ is a Leibniz algebra and its only nontrivial proper ideal is $Leib({\A}) = V(m)$. Hence $\A$ is a simple Leibniz algebra. 

Now consider the $sl(2, \mathbb{C})$ - module $V = V(m) \oplus V(n)$ where $V(m)$ and $V(n)$ are irreducible $sl(2, \mathbb{C})$-modules and $m, n \geq 1$. Consider the algebra ${\widehat \A} = sl(2, \mathbb{C}) \dotplus V$ with the left multiplication of any vector in $sl(2, \mathbb{C})$ with a vector in $V$ given by the module action and the right multiplication being trivial. Then ${\widehat \A}$ is a Leibniz algebra (see \cite{sl2}). Then $V= Leib({\widehat \A})$, $V(m)$, and $V(n)$ are ideals of ${\widehat \A}$. Hence ${\widehat \A}$ not a simple Leibniz algebra although ${\widehat \A}/Leib({\widehat \A})$ is a simple Lie algebra. Furthermore, observe that ${\widehat \A}$ can not be written as direct sum of simple Leibniz ideals.
\end{ex}

\begin{thm}\label{directsum}
Suppose $\A$ is a semisimple Leibniz algebra over $\ff$. Then 
$$A=(S_1\oplus S_2\oplus \cdots \oplus S_k)\dotplus Leib(A),$$
where $S_j$ is a simple Lie algebra for all $1 \leq j \leq k$.
\end{thm}
\begin{proof}
By Theorem \ref{levi}, ${\A}=S\dotplus R$, where $S$ is a semisimple Lie algebra and $R = rad({\A}) = Leib({\A})$. Since $S$ is
a semisimple Lie algebra we have $S = S_1\oplus S_2\oplus \cdots \oplus S_k$ where each $S_j$ is a simple ideal of $S$.

\end{proof}

The following is immediate from the above theorem.

\begin{cor} If $\A$ is a semisimple Leibniz algebra then $[{\A} , {\A}] = {\A}$.
\end{cor}

For a Leibniz algebra $\A$, we define $\kappa ( \, , \, ) : {\A} \times {\A} \longrightarrow {\A}$ by $\kappa (a , b) = tr(L_aL_b)$ for all $a , b \in \A$.
Then $\kappa ( \, , \, )$ is an invariant symmetric bilinear form on $\A$ which we call the Killing form. Note that if $\A$ is a Lie algebra then $\kappa ( \, , \, )$ coincides with the Killing form. As usual we define the radical of $\kappa ( \, , \, )$ by 
$${\A}^{\perp}=\{b\in {\A} \mid \kappa (b, a)=0 \ \ {\rm for \, \,  all} \ \ a\in {\A}\}.$$
Since the form $\kappa ( \, , \, )$ is invariant (i.e. $\kappa( [a, b], c)=\kappa (a, [b, c])$) the radical of the form ${\A}^{\perp}$ is an ideal of $\A$. It is also clear that $Leib({\A})\subseteq {\A}^{\perp}$.

\begin{defn} 
Let $\A$ be a Leibniz algebra. The Killing form $\kappa ( \, , \, )$ on $\A$ is said to be nondegenerate if ${\A}^{\perp}=Leib({\A})$.
\end{defn}

Suppose $\A$ be a nilpotent Leibniz algebra. Then by Engel's theorem (see Corrolary \ref{Engel2} ) the left multiplication operators $\{L_a \mid a \in {\A}\}$ can be simultaneously strictly upper triangularized. Hence the Killing form  $\kappa ( \, , \, )$ is trivial (i.e. $\kappa (a  , b) = 0$ for all $a, b \in \A$).

\begin{ex} Let ${\A} = {\rm span}\{h, e, f, x_0, x_1\}$ be a $5$-dimensional Leibniz algebra with the nontrivial multiplications given by:
\begin{equation*}
\begin{cases}
[h, e]=2e, [h, f]=-2f, [e, f]=h, [e, h]=-2e, [f, h]=2f, [f, e]=-h\\ 
[h, x_0]=x_0, [f, x_0]=x_1, [h, x_1]=-x_1, [e, x_1]=-x_0. 
\end{cases}
\end{equation*}
Then $rad({\A}) = Leib({\A}) = {\rm span}\{x_0, x_1\}$. Hence $\A$ is a semisimple Leibniz algebra. Also Killing form $\kappa ( \, , \, )$ on $\A$ is nondegenerate, since ${\A}^{\perp}=Leib({\A})$.
\end{ex}
\begin{thm} 
Let $\A$ be a semisimple Leibniz algebra. Then the Killing form $\kappa ( \, , \, )$ on $\A$ is nondegenerate.
\end{thm}
\begin{proof} Let $\A$ be a semisimple Leibniz algebra over $\ff$. Then by Theorem \ref{levi},
${\A} = S\dotplus Leib({\A})$ where $S$ is a semisimple Lie algebra. Let $x\in{\A}^{\perp}$, so $x$ can be written as $x=s+u$ for $s\in S, u\in Leib({\A})$. Since $S$ is a semisimple Lie algebra we have $S^{\perp}=0$. Suppose $s \neq 0$, then $s\notin S^{\perp}$. Hence by definition $\exists t\in S$ such that $tr_S(L_sL_t)=tr_{\A}(L_sL_t)\neq 0$. However, we know that $x\in {\A}^{\perp}$, and so $tr_{\A}(L_xL_a)=0$ for all $a\in \A$. In particular, $tr_{\A}(L_xL_t)=0$.
\begin{center}
$0=tr_{\A}(L_xL_t)=tr_{\A}(L_{s+u}L_t)=tr_{\A}(L_sL_t)+tr_{\A}(L_uL_t)=tr_A(L_sL_t),$
\end{center} 
since $L_uL_t = 0$. This is a contradiction. Hence $s = 0$, which implies $x\in Leib({\A})$. Thus, ${\A}^{\perp}=Leib({\A})$ and therefore the Killing form $\kappa ( \, , \, )$ on $\A$ is nondegenerate.
\end{proof}

However, unlike in case of Lie algebras as the example below shows the Killing form $\kappa ( \, , \, )$ on $\A$ being nondegenerate does not imply that $\A$ is a semisimple Leibniz algebra.
\begin{ex} Let ${\A} = {\rm span}\{x , y\}$ be the $2$-dimensional Leibniz algebra with the nontrivial multiplications given by:
\begin{center} $[y, x]=x, [y, y]=x$
\end{center}
Then ${\A}^{\perp}= {\rm span}\{x\}=Leib({\A})$. Hence the Killing form on $\A$ is nondegenerate. However, since $\A$ is solvable $rad({\A}) = \A$. Therefore,  $\A$ is not semisimple.
\end{ex}

\section{Classification of Low dimensional Leibniz algebras}

The classification of Leibniz algebras is still an open problem. So far the complete classification of Leibniz algebras of dimension less than or equal to three is known (see  \cite{omirov1998},  \cite{3dim}, \cite{CILL}, \cite{lodayfr}, \cite{3comp}) and partial results are known for dimension four (see \cite{fourdimnil}, \cite{fourdim}). In this paper we revisit the classification of non-Lie Leibniz algebras of dimension less than or equal to three. 

Suppose $\A$ be a non-Lie Leibniz algebra over $\ff$. Then $Leib({\A}) \neq 0$ and $Leib({\A}) \neq {\A}$. So there does not exist any non-Lie Leibniz algebra with ${\rm dim}({\A}) = 1$. Hence ${\rm dim}({\A}) \geq 2$. Now assume ${\rm dim}({\A}) = 2$. Since $Leib({\A}) \neq 0$, there exists $0\neq a \in \A$ such that $a^2 \neq 0$. Since $Leib({\A})$ is one-dimensional in this case, we have $Leib({\A}) = {\rm span}\{a^2\}$. Then $[a^2, a^2] = 0$ and $[a, a^2] = \alpha a^2$ for some $\alpha \in \ff$ since $Leib({\A})$ is an abelian ideal. Thus ${\A} = {\rm span}\{a, a^2\}$ and we have two possibilities: $\alpha = 0$ or $\alpha \neq 0$. If $\alpha = 0$, then $[a, a^2] = 0$ and $\A$ is a nilpotent cyclic Leibniz algebra generated by $a$. If $\alpha \neq 0$, then replacing $a$ by $\frac{1}{\alpha} a$ we see that $[a, a^2] = a^2$ and $\A$ is a solvable cyclic Leibniz algebra generated by $a$. Thus we have the following theorem.

\begin{thm}Let $\A$ be a non-Lie Leibniz algebra and ${\rm dim}({\A}) = 2$. Then $\A$ is isomorphic to a cyclic Leibniz algebra generated by $a$ with either $[a, a^2] = 0$ (hence $\A$ is nilpotent) or $[a, a^2] = a^2$ (hence $\A$ is solvable).
\end{thm}

\begin{thm} Let $\A$ be a non-Lie Leibniz algebra and ${\rm dim}({\A}) \leq 4$. Then $\A$ is solvable.
\end{thm}

\begin{proof} Since $\A$ is a non-Lie Leibniz algebra, $Leib({\A}) \neq 0$. Hence ${\dim}(Leib({\A}))$ $ \geq 1$ and ${\dim}({\A}/Leib({\A})) \leq 3$. Then the Lie algebra ${\A}/Leib({\A})$ is either solvable or simple. If ${\A}/Leib({\A})$ is solvable, then $\A$ is solvable since $Leib({\A})$ is an abelian ideal. If ${\A}/Leib({\A})$ is a simple Lie algebra then ${\rm dim}({\A}/Leib({\A})) = 3$ and it is isomorphic to $s\ell(2, \mathbb{C})$. By Theorem \ref{levi}, we have ${\A}=S\dotplus Leib({\A})$ where $S$ is a subalgebra of $\A$ and as a Lie algebra it is isomorphic to $s\ell(2, \mathbb{C})$. Since ${\dim}(Leib({\A})) = 1$, the Lie algebra $S$ acts trivially on $Leib({\A})$ which implies $[S , Leib({\A})] = 0$. Since 
$[Leib({\A}), S] = 0$, it follows that $\A$ is a Lie algebra which is a contradiction.
\end{proof}

The following result from \cite{vgor} is also useful.
\begin{prop} \label{codim1} (\cite{vgor} , {\rm Corollary 3})  If the Leibniz algebra $\A$ is nilpotent of dimension $n$ and ${\rm dim}([{\A} , {\A}]) = n-1$, then $\A$ is a cyclic Leibniz algebra generated by a single element.
\end{prop}

\begin{thm}\label{nilclass} Let $\A$ be a non-Lie nilpotent Leibniz algebra and ${\rm dim}({\A}) = 3$. Then $\A$ is isomorphic to a Leibniz algebra spanned by $\{x, y, z\}$ with the nonzero products given by one of the following:
\begin{enumerate}
\item $[x, x] = y, [x, y] = z $.
\item $[x, x] = z$ .
\item $[x, y] = z, [y, x] = z$.
\item $[x, y] = z, [y, x] = -z, [y, y] = z$.
\item $[x, y] = z, [y, x] = \alpha z, \ \  \alpha \in \mathbb{C} \setminus \{1, -1\}$.
\end{enumerate}
\end{thm}

\begin{proof} Since $\A$ is nilpotent ${\A}^2 \neq {\A}$. Also ${\A}^2 \supseteq Leib({\A}) \neq 0$. So ${\rm dim}({\A}^2) = 2 \ \ {\rm or} \ \ 1$.
If ${\rm dim}({\A}^2) = 2$, then by Proposition \ref{codim1} $\A$ is generated by a single element $x \in {\A}$. Take $[x, x] = y, [x, y] = z$. Since $\A$ is nilpotent we now have $[x, z] = 0$.

Now suppose ${\rm dim}({\A}^2) = 1$. Then ${\A}^2 = Leib({\A})$ and ${\A}^2 = {\rm span}\{z\}$ for some $0 \neq z \in {\A}$. Let $V$ be a complementary subspace to ${\A}^2$ in $\A$. Then for any $u, v \in V$, we have $[u, v] = cz$ for some $c \in \mathbb{C}$. Define the bilinear form $f( \ , \ ) : V \times V \longrightarrow \mathbb{C}$ by $f(u, v) = c$ for all $u, v \in V$. The canonical forms for the congruence classes of matrices associated with the bilinear form $f( \ , \  )$ as given in \cite{TA} are listed below:
\vskip 5pt
\noindent $(i) \left( \begin{array}{cc}
0&1 \\
-1&0
\end{array} \right),
(ii) \left( \begin{array}{cc}
1&0 \\
0&0
\end{array} \right), 
(iii) \left( \begin{array}{cc}
0&1 \\
1&0
\end{array} \right),
(iv) \left( \begin{array}{cc}
0&1 \\
-1&1
\end{array} \right),
(v) \left( \begin{array}{cc}
0&1 \\
c&0
\end{array} \right)$, 

\noindent ${\rm where} \ \ c \neq 1, -1$.

Now we choose an ordered basis $\{x, y\}$ for $V$ with respect to which the matrix of the bilinear form $f( \ , \ )$ is one of the above. Then ${\A} = {\rm span}\{x, y, z\}$ and the products among the basis vectors are completely determined by the matrix of the bilinear form $f( \ , \ )$ as shown below.

If the matrix of the bilinear form $f( \ , \ )$ is given by $(i)$, then we have the nonzero products are $[x, y] = z , [y, x] = -z$ and $\A$ is a Heisenberg Lie algebra which is a contradiction since by assumption $\A$ is not a Lie algebra. Thus the matrix can not be given by $(i)$ in this case.

If the matrix of the bilinear form $f( \ , \ )$ is given by $(ii), (iii), (iv)$ or $(v)$, it is easy to see that the nonzero products are indeed given by equations $(2), (3), (4)$ or $(5)$ in the statement of the theorem respectively which completes the proof of the theorem.
\end{proof}

\begin{rmk} We observe that by suitable change of basis the isomorphism classes in Theorem \ref{nilclass} coincide with the isomorphism classes given in (\cite{CILL}, page 3752) as follows: $(1) \rightarrow 3(c), (2) \rightarrow 2(b), (3) \rightarrow 2(c), (4) \rightarrow 2(a) \ \ {\rm with}  \ \ \gamma = \frac{1}{4},
(5) \rightarrow 2(a) \ \ {\rm with} \ \ \gamma \neq \frac{1}{4}$. We also observe that our approach differs from the computational approach in  \cite{omirov1998}, \cite{3dim}, \cite{3comp} and \cite{CILL}. Indeed our approach can easily be adapted to classify all finite dimensional non-Lie nilpotent Leibniz algebras $\A$ with ${\rm codim}({\A}^2) = 1$ or ${\rm dim}({\A}^2) = 1$ . 
\end{rmk}

Finally, we state the classification of three dimensional non-Lie, solvable Leibniz algebra $\A$ which is not nilpotent. In this case ${\A}^2$ is nilpotent, hence the nilradical contains ${\A}^2$ and the isomorphism classes as given in (\cite{CILL}, page 3752) (see also \cite{omirov1998}, \cite{3dim}, \cite{3comp}) are as follows. (Note that in \cite{CILL} the classification is done for right Leibniz algebras).

\begin{thm}Let $\A$ be a non-Lie non-nilpotent solvable Leibniz algebra and ${\rm dim}({\A}) = 3$. Then $\A$ is isomorphic to a Leibniz algebra spanned by $\{x, y, z\}$ with the nonzero products given by one of the following:
\begin{enumerate}
\item $[z, x] = x $.
\item $[z, x] = \alpha x, \alpha \in \mathbb{C}\setminus \{0\};  [z, y] = y;  [y, z] = -y$.
\item $[z, y] = y; [y, z] = -y; [z, z] = x$.
\item $[z, x] = 2x; [y, y] = x; [z, y] = y; [y, z] = -y; [z, z] = x$.
\item $[z, y] = y; [z, x] = \alpha x, \ \  \alpha \in \mathbb{C} \setminus \{0\}$.
\item $[z, x] = x+y; [z, y] = y$.
\item $[z, x] = y; [z, y] = y; [z, z] = x$.
\end{enumerate}

\end{thm}

\bibliographystyle{amsplain}

\end{document}